\documentclass[fleqn,preprint,3p,a4paper]{elsarticle}

\usepackage{amssymb}
\usepackage{amsmath}
\usepackage{amsthm}
\usepackage{dcolumn}
\usepackage{endnotes}
\usepackage{tabularx}
\usepackage[matrix,arrow]{xy}
\usepackage{wasysym}

\newtheorem{theorem}{Theorem}[section]
\newtheorem{proposition}[theorem]{Proposition}
\newtheorem{lemma}[theorem]{Lemma}
\newtheorem{corollary}[theorem]{Corollary}

\theoremstyle{definition}
\newtheorem{definition}[theorem]{Definition}
\newtheorem{example}[theorem]{Example}
\newtheorem{remark}[theorem]{Remark}

\newtheorem{problem}[theorem]{Problem}

\newcommand{\ir}{{\mathsf{Irr}}}

\newcommand{\mn}{\mathbb N}

\newcommand{\ii}{{\rm int}}

\newcommand{\ua}{\mathord{\uparrow}}
\newcommand{\da}{\mathord{\downarrow}}

\newcommand{\mk}{\mathord{\mathsf{K}}}
\newcommand{\wdd}{\mathord{\mathsf{WD}}}

\begin{document}

\begin{frontmatter}



\title{$\omega$-Rudin spaces, well-filtered determined spaces and first-countable spaces\tnoteref{t1}}
\tnotetext[t1]{This research was supported by the National Natural Science Foundation of China (Nos. 11661057, 11361028,
61300153, 11671008, 11701500, 11626207); the Natural Science Foundation of Jiangxi Province, China (No. 20192ACBL20045); NSF Project of Jiangsu Province,
China (BK20170483); and NIE ACRF (RI 3/16 ZDS), Singapore}

\author[X. Xu]{Xiaoquan Xu\corref{mycorrespondingauthor}}
\cortext[mycorrespondingauthor]{Corresponding author}
\ead{xiqxu2002@163.com}
\address[X. Xu]{School of Mathematics and Statistics,
Minnan Normal University, Zhangzhou 363000, China}
\author[C. Shen]{Chong Shen}
\address[C. Shen]{School of Mathematical Sciences, Nanjing Normal University, Nanjing 210046, China}
\ead{shenchong0520@163.com}
\author[X. Xi]{Xiaoyong Xi}
\ead{littlebrook@jsnu.edu.cn}
\address[X. Xi]{School of mathematics and Statistics,
Jiangsu Normal University, Xuzhou 221116, China}
\author[D. Zhao]{Dongsheng Zhao}
\address[D. Zhao]{Mathematics and Mathematics Education,
National Institute of Education Singapore, \\
Nanyang Technological University,
1 Nanyang Walk, Singapore 637616}
\ead{dongsheng.zhao@nie.edu.sg}

\begin{abstract}
  We investigate some versions of $d$-space, well-filtered space and Rudin space concerning various countability properties. The main results include: (i) if the sobrification of a $T_0$ space $X$ is first-countable, then $X$ is an $\omega$-Rudin space; (ii) every $\omega$-well-filtered space is sober if its sobrification is first-countable; (iii) if a $T_0$ space is second-countable or first-countable and with a countable underlying set, then it is a $\omega$-Rudin space; (iv) every first-countable $T_0$ space is well-filtered determined; (v) every irreducible closed subset in a first-countable $\omega$-well-filtered space is countably-directed; (vi) every first-countable $\omega$-well-filtered $\omega^\ast$-$d$-space is sober.
\end{abstract}





\end{frontmatter}


\section{Introduction}

In \cite{Shenchon, xu-shen-xi-zhao2, xu-shen-xi-zhao1}, we introduced and studied the Rudin spaces, well-filtered determined spaces and $\omega$-well-filtered spaces. Some relationships and links  among these new non-Hausdorff topological properties and the well studied sobriety and well-filteredness were uncovered. In \cite{xu-shen-xi-zhao1}, it was proved that in a first-countable $\omega$-well-filtered space $X$, every irreducible closed subset of $X$ is directed under the specialization order of $X$. It follows immediately that every first-countable $\omega$-well-filtered $d$-space is sober.

In the current paper, we continue studying some aspects of $d$-space, well-filtered space and Rudin spaces concerning countability. Employing countably-directed sets, we define the $\omega^\ast$-$d$-spaces and $\omega^\ast$-well-filtered spaces. It is proved that if the sobrification of a $T_0$ space $X$ is first-countable, then $X$ is an $\omega$-Rudin space. Therefore, every $\omega$-well-filtered space is sober if it has a first countable sobrification. From these, we obtain that if a $T_0$ space $X$ is second-countable or first-countable with a countable underlying set, then $X$ is an $\omega$-Rudin space, and $X$ is sober if it is additionally $\omega$-well-filtered. Another major result obtained is that every first-countable $T_0$ space is well-filtered determined. In each first-countable $\omega$-well-filtered space, every irreducible closed subset is proved to be countably-directed, hence every first-countable $\omega$-well-filtered $\omega^\ast$-$d$-space is sober. We also prove that a $T_0$ space $Y$ is $\omega^\ast$-well-filtered iff its Smyth power space is $\omega^\ast$-well-filtered iff its Smyth power space is an $\omega^\ast$-$d$-space. The work presented here enriched the theory of non-Hausdorff topological spaces and lead to some nontrivial open problems for further investigation.

\section{Preliminary}

In this section, we briefly recall some fundamental concepts and notations that will be used in the paper. Some basic properties of irreducible sets and compact saturated sets are presented.

For a poset $P$ and $A\subseteq P$, let
$\mathord{\downarrow}A=\{x\in P: x\leq  a \mbox{ for some }
a\in A\}$ and $\mathord{\uparrow}A=\{x\in P: x\geq  a \mbox{
	for some } a\in A\}$. For  $x\in P$, we write
$\mathord{\downarrow}x$ for $\mathord{\downarrow}\{x\}$ and
$\mathord{\uparrow}x$ for $\mathord{\uparrow}\{x\}$. A subset $A$
is called a \emph{lower set} (resp., an \emph{upper set}) if
$A=\mathord{\downarrow}A$ (resp., $A=\mathord{\uparrow}A$). Let $P^{(<\omega)}=\{F\subseteq P : F \mbox{~is a nonempty finite set}\}$. For a set $X$, $|X|$ will denote the cardinality  of $X$. Let $\mn$ denotes the set of all natural numbers with the usual order and $\omega=|\mn|$.

A nonempty subset $D$ of a poset $P$ is \emph{directed} if every two
elements in $D$ have an upper bound in $D$. The set of all directed sets of $P$ is denoted by $\mathcal D(P)$. A subset $I\subseteq P$ is called an \emph{ideal} of $P$ if $I$ is a directed lower subset of $P$. Let $\mathrm{Id} (P)$ be the poset (with the order of set inclusion) of all ideals of $P$. Dually, we define the notion of \emph{filters} and denote the poset of all filters of $P$ by $\mathrm{Filt}(P)$.  A poset $P$ is called a
\emph{directed complete poset}, or \emph{dcpo} for short, if for any
$D\in \mathcal D(P)$, $\bigvee D$ exists in $P$.

As in \cite{redbook}, the \emph{upper topology} on a poset $Q$, generated
by the complements of the principal ideals of $Q$, is denoted by $\upsilon (Q)$. A subset $U$ of $Q$ is \emph{Scott open} if
(i) $U=\mathord{\uparrow}U$ and (ii) for any directed subset $D\subseteq Q$ with
$\bigvee D$ existing, $\bigvee D\in U$ implies $D\cap
U\neq\emptyset$. All Scott open subsets of $Q$ form a topology,
called the \emph{Scott topology} on $Q$ and
denoted by $\sigma(Q)$. The space $\Sigma~\!\! Q=(Q,\sigma(Q))$ is called the
\emph{Scott space} of $Q$. The upper sets of $Q$ form the (\emph{upper}) \emph{Alexandroff topology} $\alpha (Q)$.

For a $T_0$ space $X$, we use $\leq_X$ to denote the \emph{specialization order} on $X$: $x\leq_X y$ if{}f $x\in \overline{\{y\}}$). In the following, when a $T_0$ space $X$ is considered as a poset, the partial order always means the specialization order provided otherwise indicated. Let $\mathcal O(X)$ (resp., $\mathcal C(X)$) be the set of all open subsets (resp., closed subsets) of $X$, and let $\mathcal S^u(X)=\{\ua x : x\in X\}$. Let $\mathcal S_c(X)=\{\overline{{\{x\}}} : x\in X\}$ and $\mathcal D_c(X)=\{\overline{D} : D\in \mathcal D(X)\}$.

A nonempty subset $A$ of $X$ is \emph{irreducible} if for any $\{F_1, F_2\}\subseteq \mathcal C(X)$, $A \subseteq F_1\cup F_2$ implies $A \subseteq F_1$ or $A \subseteq  F_2$.  Denote by $\ir(X)$ (resp., $\ir_c(X)$) the set of all irreducible (resp., irreducible closed) subsets of $X$. Every directed subset of $X$ is irreducible. $X$ is called \emph{sober}, if for any  $F\in\ir_c(X)$, there is a unique point $a\in X$ such that $F=\overline{\{a\}}$.

\begin{remark}\label{A is closure of x then sup A is x} In a $T_0$ space $X$, if $x\in X$ and $A\subseteq X$ such that $\overline{A}=\overline{\{x\}}$, then $\bigvee A$ exists in $X$ and $x=\bigvee A$.
\end{remark}

The following two lemmas on irreducible sets are well-known.

\begin{lemma}\label{irrsubspace}
Let $X$ be a space and $Y$ a subspace of $X$. Then the following conditions are equivalent for a
subset $A\subseteq Y$:
\begin{enumerate}[\rm (1)]
	\item $A$ is an irreducible subset of $Y$.
	\item $A$ is an irreducible subset of $X$.
	\item ${\rm cl}_X A$ is an irreducible subset of $X$.
\end{enumerate}
\end{lemma}

\begin{lemma}\label{irrimage}
	If $f : X \longrightarrow Y$ is continuous and $A\in\ir (X)$, then $f(A)\in \ir (Y)$.
\end{lemma}

A $T_0$ space $X$ is called a \emph{d-space} (or \emph{monotone convergence space}) if $X$ (with the specialization order) is a dcpo
 and $\mathcal O(X) \subseteq \sigma(X)$ (cf. \cite{redbook, Wyler}).

\begin{definition}\label{DCspace} (\cite{xu-shen-xi-zhao1})
	A $T_0$ space $X$ is called a \emph{directed closure space}, $\mathsf{DC}$ \emph{space} for short, if $\ir_c(X)=\mathcal{D}_c(X)$, that is, for each $A\in \ir_c(X)$, there exists a directed subset of $X$ such that $A=\overline{D}$.
\end{definition}

For any topological space $X$, $\mathcal G\subseteq 2^{X}$ and $A\subseteq X$, let $\Diamond_{\mathcal G} A=\{G\in \mathcal G : G\bigcap A\neq\emptyset\}$ and $\Box_{\mathcal G} A=\{G\in \mathcal G : G\subseteq  A\}$. The symbols $\Diamond_{\mathcal G} A$ and $\Box_{\mathcal G} A$ will be simply written as $\Diamond A$  and $\Box A$ respectively if no ambiguity occurs. The \emph{lower Vietoris topology} on $\mathcal{G}$ is the topology that has $\{\Diamond U : U\in \mathcal O(X)\}$ as a subbase, and the resulting space is denoted by $P_H(\mathcal{G})$. If $\mathcal{G}\subseteq \ir (X)$, then $\{\Diamond_{\mathcal{G}} U : U\in \mathcal O(X)\}$ is a topology on $\mathcal{G}$. The space $P_H(\mathcal{C}(X)\setminus \{\emptyset\})$ is called the \emph{Hoare power space} or \emph{lower space} of $X$ and is denoted by $P_H(X)$ for short (cf. \cite{Schalk}). Clearly, $P_H(X)=(\mathcal{C}(X)\setminus \{\emptyset\}, \upsilon(\mathcal{C}(X)\setminus \{\emptyset\}))$. So $P_H(X)$ is always sober (see \cite[Corollary 4.10]{ZhaoHo} or \cite[Proposition 2.9]{xu-shen-xi-zhao1}). The \emph{upper Vietoris topology} on $\mathcal{G}$ is the topology that has $\{\Box_{\mathcal{G}} U : U\in \mathcal O(X)\}$ as a base, and the resulting space is denoted by $P_S(\mathcal{G})$.

\begin{remark} \label{eta continuous} Let $X$ be a $T_0$ space.
\begin{enumerate}[\rm (1)]
	\item If $\mathcal{S}_c(X)\subseteq \mathcal{G}$, then the specialization order on $P_H(\mathcal{G})$ is the set inclusion order, and the \emph{canonical mapping} $\eta_{X}: X\longrightarrow P_H(\mathcal{G})$, given by $\eta_X(x)=\overline {\{x\}}$, is an order and topological embedding (cf. \cite{redbook, Jean-2013, Schalk}).
    \item The space $X^s=P_H(\ir_c(X))$ with the canonical mapping $\eta_{X}: X\longrightarrow X^s$ is the \emph{sobrification} of $X$ (cf. \cite{redbook, Jean-2013}).
\end{enumerate}
\end{remark}

 A subset $A$ of a space $X$ is called \emph{saturated} if $A$ equals the intersection of all open sets containing it (equivalently, $A$ is an upper set in the specialization order). We shall use $\mathord{\mathsf{K}}(X)$ to
denote the set of all nonempty compact saturated subsets of $X$ and endow it with the \emph{Smyth preorder}, that is, for $K_1,K_2\in \mathord{\mathsf{K}}(X)$, $K_1\sqsubseteq K_2$ if{}f $K_2\subseteq K_1$. The space $P_S(\mathord{\mathsf{K}}(X))$, denoted shortly by $P_S(X)$, is called the \emph{Smyth power space} or \emph{upper space} of $X$ (cf. \cite{Heckmann, Schalk}). It is easy to verify that the specialization order on $P_S(X)$ is the Smyth order (that is, $\leq_{P_S(X)}=\sqsubseteq$). The \emph{canonical mapping} $\xi_X: X\longrightarrow P_S(X)$, $x\mapsto\ua x$, is an order and topological embedding (cf. \cite{Heckmann, Klause-Heckmann, Schalk}). Clearly, $P_S(\mathcal S^u(X))$ is a subspace of $P_S(X)$ and $X$ is homeomorphic to $P_S(\mathcal S^u(X))$.

\begin{lemma}\label{sups in Smyth}\emph{(\cite{redbook})}  Let $X$ be a $T_0$ space. For any nonempty family $\{K_i : i\in I\}\subseteq \mk (X)$, $\bigvee_{i\in I} K_i$ exists in $\mk (X)$ if{}f~$\bigcap_{i\in I} K_i\in \mk (X)$. In this case $\bigvee_{i\in I} K_i=\bigcap_{i\in I} K_i$.
\end{lemma}

\begin{lemma}\label{K union} \emph{(\cite{jia-Jung-2016, Schalk})}  Let $X$ be a $T_0$ space.
\begin{enumerate}[\rm (1)]
\item If $\mathcal K\in\mk(P_S(X))$, then  $\bigcup \mathcal K\in\mk(X)$.
\item The mapping $\bigcup : P_S(P_S(X)) \longrightarrow P_S(X)$, $\mathcal K\mapsto \bigcup \mathcal K$, is continuous.
\end{enumerate}
\end{lemma}

A $T_0$ space $X$ is called \emph{well-filtered} if it is $T_0$, and for any open set $U$ and filtered family $\mathcal{K}\subseteq \mathord{\mathsf{K}}(X)$, $\bigcap\mathcal{K}{\subseteq} U$ implies $K{\subseteq} U$ for some $K{\in}\mathcal{K}$.

\begin{remark}\label{sober implies WF implies d-space} The following implications are well-known (cf. \cite{redbook}):

$$\mbox{sobriety $\Rightarrow$ well-filteredness $\Rightarrow$ $d$-space.}$$
\end{remark}

\section{Topological Rudin's Lemma, Rudin spaces and well-filtered determined spaces}

Rudin's Lemma is a useful tool in topology and plays a crucial role in domain theory (see [1, 3-9, 20-21, 23]). Heckmann and Keimel \cite{Klause-Heckmann} presented the following topological variant of Rudin's Lemma.

\begin{lemma}\label{t Rudin} \emph{(Topological Rudin's Lemma)} Let $X$ be a topological space and $\mathcal{A}$ an
irreducible subset of the Smyth power space $P_S(X)$. Then every closed set $C {\subseteq} X$  that
meets all members of $\mathcal{A}$ contains a minimal irreducible closed subset $A$ that still meets all
members of $\mathcal{A}$.
\end{lemma}

Applying Lemma \ref{t Rudin} to the Alexandroff topology on a poset $P$, one obtains the original Rudin's Lemma (see \cite{Rudin}).

\begin{corollary}\label{rudin} \emph{(Rudin's Lemma)} Let $P$ be a poset, $C$ a nonempty lower subset of $P$ and $\mathcal F\in \mathbf{Fin} P$ a filtered family with $\mathcal F\subseteq\Diamond C$. Then there exists a directed subset $D$ of $C$ such that $\mathcal F\subseteq \Diamond\da D$.
\end{corollary}

For a $T_0$ space $X$ and $\mathcal{K}\subseteq \mathord{\mathsf{K}}(X)$, let $M(\mathcal{K})=\{A\in \mathcal C(X) : K\bigcap A\neq\emptyset \mbox{~for all~} K\in \mathcal{K}\}$ (that is, $\mathcal K\subseteq \Diamond A$) and $m(\mathcal{K})=\{A\in \mathcal C(X) : A \mbox{~is a minimal member of~} M(\mathcal{K})\}$.

By the proof of \cite[Lemma 3.1]{Klause-Heckmann}, we have the following result.

\begin{lemma}\label{t ruding}  Let $X$ be a $T_0$ space and $\mathcal{K}\subseteq \mathord{\mathsf{K}}(X)$. If $C\in M(\mathcal{K})$, then there is a closed subset $A$ of $C$ such that $C\in m(\mathcal{K})$.
\end{lemma}

In \cite{Shenchon, xu-shen-xi-zhao1}, based on topological Rudin's Lemma, Rudin spaces and well-filtered determined spaces were introduced and studied.

\begin{definition}\label{rudinset} (\cite{Shenchon, xu-shen-xi-zhao1})
		Let $X$ be a $T_0$ space. A nonempty subset  $A$  of $X$  is said to have the \emph{Rudin property}, if there exists a filtered family $\mathcal K\subseteq \mathord{\mathsf{K}}(X)$ such that $\overline{A}\in m(\mathcal K)$ (that is,  $\overline{A}$ is a minimal closed set that intersects all members of $\mathcal K$). Let $\mathsf{RD}(X)=\{A\in \mathcal C(X) : A\mbox{~has Rudin property}\}$. The sets in $\mathsf{RD}(X)$ will also be called \emph{Rudin sets}. The space $X$ is called a \emph{Rudin space}, $\mathsf{RD}$ \emph{space} for short, if $\ir_c(X)=\mathsf{RD}(X)$, that is, every irreducible closed set of $X$ is a Rudin set. The category of all Rudin spaces with continuous mappings is denoted by $\mathbf{Top}_{r}$.
\end{definition}

\begin{definition}\label{WDspace} (\cite{xu-shen-xi-zhao1})
	 A subset $A$ of a $T_0$ space $X$ is called a \emph{well-filtered determined set}, $\wdd$ \emph{set} for short, if for any continuous mapping $ f:X\longrightarrow Y$
to a well-filtered space $Y$, there exists a unique $y_A\in Y$ such that $\overline{f(A)}=\overline{\{y_A\}}$.
Denote by $\mathsf{WD}(X)$ the set of all closed well-filtered determined subsets of $X$. $X$ is called a \emph{well-filtered determined} space, $\mathsf{WD}$ \emph{space} for short, if all irreducible closed subsets of $X$ are well-filtered determined, that is, $\ir_c(X)=\wdd (X)$.
\end{definition}

\begin{proposition}\label{DRWIsetrelation}(\cite{xu-shen-xi-zhao1})
	Let $X$ be a $T_0$ space. Then $\mathcal{D}_c(X)\subseteq \mathsf{RD}(X)\subseteq\mathsf{WD}(X)\subseteq\ir_c(X)$.
\end{proposition}

\begin{corollary}\label{SDRWspacerelation}(\cite{xu-shen-xi-zhao1})
	Sober $\Rightarrow$ $\mathsf{DC}$ $\Rightarrow$ $\mathsf{RD}$ $\Rightarrow$ $\mathsf{WD}$.
\end{corollary}

 A topological space $X$ is \emph{locally hypercompact} if for each $x\in X$ and each open neighborhood $U$ of $x$, there is  $\ua F\in \mathbf{Fin}X$ such that $x\in\ii\,\ua F\subseteq\ua F\subseteq U$ (cf. \cite{E_20182}). A space $X$ is called \emph{core}-\emph{compact} if $(\mathcal O(X), \subseteq)$ is a \emph{continuous lattice} (cf. \cite{redbook}).

\begin{theorem}\label{LHCdirected} (\cite{E_20182})
	Let $X$ be a locally hypercompact $T_0$ space and $A\in\ir(X)$. Then there exists a directed subset $D\subseteq\da A$ such that $\overline{A}=\overline{D}$. Therefore, $X$ is a $\mathsf{DC}$ space, and hence a Rudin space and a $\wdd$ space.
\end{theorem}

\begin{theorem}\label{LCrudin} (\cite{xu-shen-xi-zhao1})
	Every locally compact $T_0$ space is a Rudin space.
\end{theorem}

\begin{theorem}\label{CorecomptWD} (\cite{xu-shen-xi-zhao1}) Every core-compact $T_0$ space is well-filtered determined.
\end{theorem}

By Theorem \ref{CorecomptWD}, we immediately deduce the following.

\begin{corollary}\label{corecwellf sober}  (\cite{Lawson-Xi, xu-shen-xi-zhao1}) Every core-compact well-filtered space is sober.
\end{corollary}

At the moment, it is still not sure wether every core-compact $T_0$ space is a Rudin space.

\section{$\omega$-$d$-Spaces and $\omega$-well-filtered spaces}

For a $T_0$ space $X$, let $\mathcal D^\omega(X)=\{D\subseteq X : D \mbox{ is countable and directed}\}$ and $\mathcal D^\omega_c(X)=\{\overline{D} : D\in \mathcal D^\omega(X)\}$.

\begin{definition}\label{omega dcpo}(\cite{xu-shen-xi-zhao2})
	A poset $P$ is called an \emph{$\omega$-dcpo}, if for any $D\in \mathcal D^\omega(P)$, $\bigvee D$ exists.
\end{definition}

\begin{lemma}\label{cchain} (\cite{xu-shen-xi-zhao2})
Let $P$ be a poset and $D\in \mathcal D^\omega(P)$. Then there exists a countable chain $C\subseteq D$ such that $D=\da C$. Hence, $\bigvee C$ exists and $\bigvee C=\bigvee D$ whenever $\bigvee D$ exists.
\end{lemma}

By Lemma \ref{cchain}, a poset $P$ is an $\omega$-dcpo if{}f for any countable chain $C$ of $P$, $\bigvee C$ exists.

\begin{definition}\label{omega Scott topology} (\cite{xu-shen-xi-zhao2})
	Let $P$ be a poset. A subset $U$ of $P$ is called \emph{$\omega$-Scott open} if (i) $U=\ua U$, and (ii) for any countable directed set $D$, $\bigvee D\in U$ implies that $D\cap U\neq\emptyset$. All $\omega$-Scott open sets form a topology on $P$,  denoted by $\sigma_\omega(P)$ and called the \emph{$\omega$-Scott topology}. The space $\Sigma_{\omega}P=(P,\sigma_{\omega}(P))$ is called the \emph{$\omega$-Scott space} of $P$.
\end{definition}

Clearly, $\sigma(P)\subseteq\sigma_{\omega}(P)$. The converse need not be true, see Example \ref{omega-d-space needed} in Section 7.

\begin{definition}\label{omega d-space} (\cite{xu-shen-xi-zhao2})
	A $T_0$ space $X$ is called an $\omega$-$d$-space (or an \emph{$\omega$-monotone convergence space}) if for any $D\in\mathcal D^\omega(X)$, the closure of $D$ has a generic point, equivalently, if $\mathcal D_c^\omega(X)=\mathcal S_c(X)$.
\end{definition}

Some characterizations of $\omega$-$d$-spaces were given in \cite[Proposition 3.7]{xu-shen-xi-zhao2}

\begin{definition}\label{omega-WF space} (\cite{xu-shen-xi-zhao2})
	A $T_0$ space $X$ is called \emph{$\omega$-well-filtered}, if for any countable filtered family $\{K_i:i<\omega\}\subseteq \mk (X)$ and  $U\in\mathcal O(X)$, it holds that
		$$\bigcap_{i<\omega}K_i\subseteq U \ \Rightarrow \  \exists i_0<\omega, K_{i_0}\subseteq U.$$
\end{definition}

By Lemma \ref{cchain}, we have the following result.

\begin{proposition}\label{omega WF charact}  (\cite{xu-shen-xi-zhao2})
	 A $T_0$ space $X$ is $\omega$-well-filtered if{}f for any countable descending chain $K_0\supseteq K_1\supseteq K_2\supseteq\ldots\supseteq K_n\supseteq\ldots$ of compact saturated subsets of $X$ and $U\in\mathcal O(X)$, the following implication holds:
	$$\bigcap_{i<\omega}K_i\subseteq U\ \Rightarrow\ \exists i_0<\omega, \ K_{i_0}\subseteq U.$$
\end{proposition}

It is easy to check that every $\omega$-well-filtered space is an $\omega$-$d$-space (see \cite[Proposition 3.11]{xu-shen-xi-zhao2}). In \cite{xu-shen-xi-zhao2}, we  introduced and studied two new classes of closed subsets in $T_0$ spaces - $\omega$-Rudin sets and  $\omega$-well-filtered determined closed sets. The $\omega$-Rudin sets lie between the class of all closures of countable directed subsets and that of $\omega$-well-filtered determined closed sets, and $\omega$-well-filtered determined closed sets lie between the class of all $\omega$-Rudin sets subsets and that of irreducible closed subsets.

\begin{definition}\label{omega rudinset} (\cite{xu-shen-xi-zhao2})
Let $X$ be a $T_0$ space. A nonempty subset $A$ of $X$ is said to have the $\omega$-\emph {Rudin property}, if there exists a countable filtered family $\mathcal K\subseteq \mathord{\mathsf{K}}(X)$ such that $\overline{A}\in m(\mathcal K)$ (that is,  $\overline{A}$ is a minimal closed set that intersects all members of $\mathcal K$). Let $\mathsf{RD}_\omega(X)=\{A\in \mathcal C(X) : A\mbox{~has $\omega$-Rudin property}\}$. The sets in $\mathsf{RD}_\omega(X)$ will also be called $\omega$-\emph{Rudin sets}. The space $X$ is called $\omega$-\emph{Rudin space}, if $\ir_c(X)=\mathsf{RD}_\omega(X)$, that is, all irreducible closed subsets of $X$ are $\omega$-Rudin sets.
\end{definition}

\begin{definition}\label{omega WDspace}(\cite{xu-shen-xi-zhao2})
	 A subset $A$ of a $T_0$ space $X$ is called an $\omega$-\emph{well}-\emph{filtered determined set}, $\wdd_\omega$ \emph{set} for short, if for any continuous mapping $ f:X\longrightarrow Y$
to an $\omega$-well-filtered space $Y$, there exists a unique $y_A\in Y$ such that $\overline{f(A)}=\overline{\{y_A\}}$.
Denote by $\mathsf{WD}_\omega(X)$ the set of all closed $\omega$-well-filtered determined subsets of $X$. The space $X$ is called $\omega$-\emph{well}-\emph{filtered determined}, $\omega$-$\mathsf{WD}$ \emph{space} for short, if $\ir_c(X)=\mathsf{WD}_\omega(X)$, that is, all irreducible closed subsets of $X$ are $\omega$-well-filtered determined.
\end{definition}

\begin{proposition}\label{omeg rudin is omega WF} (\cite{xu-shen-xi-zhao2})
	Let $X$ be a $T_0$ space. Then $\mathcal{S}_c(X)\subseteq \mathcal{D}_c^\omega(X)\subseteq \mathsf{RD}_\omega(X)\subseteq\mathsf{WD}_\omega(X)\subseteq\ir_c(X)$.
\end{proposition}

By Proposition \ref{omeg rudin is omega WF}, every $\omega$-Rudin space is $\omega$-well-filtered determined.

\begin{definition}\label{countable DC space}
	A $T_0$ space $X$ is called a \emph{countable directed closure space}, $\omega$-$\mathsf{DC}$ \emph{space} for short, if $\ir_c(X)=\mathcal{D}_c^{\omega}(X)$, that is, for each $A\in \ir_c(X)$, there exists a countable directed subset of $X$ such that $A=\overline{D}$.
\end{definition}

\begin{theorem}\label{sober equiv using omega RD and omega WD}  (\cite{xu-shen-xi-zhao2}) For a $T_0$ space $X$, the following conditions are equivalent:
	\begin{enumerate}[\rm (1)]
		\item $X$ is sober.
		\item $X$ is an $\omega$-$\mathsf{DC}$ and $\omega$-$d$-space.
        \item $X$ is an $\omega$-$\mathsf{DC}$ and $\omega$-well-filtered space.
		\item $X$ is an $\omega$-Rudin and $\omega$-well-filtered space.
		\item $X$ is an $\omega$-well-filtered determined and $\omega$-well-filtered space.
	\end{enumerate}
\end{theorem}

\section{$\omega^\ast$-Scott topologies and $\omega^\ast$-$d$-spaces}
We now introduce and study two new types of spaces.

\begin{definition}\label{countably-directed}  A nonempty subset $D$ of a poset $P$ is called \emph{countably-directed} if every nonempty countable subset of $D$ have an upper bound in $D$. The set of all countably-directed sets of $P$ is denoted by $\mathcal D^{\omega^{\ast}}(P)$. The poset $P$ is called a
\emph{countably-directed complete poset}, or $\omega^{\ast}$-\emph{dcpo} for short, if for any
$D\in \mathcal D^{\omega^\ast}(P)$, $\bigvee D$ exists in $P$.
\end{definition}

Clearly, $\{\downarrow x : x\in P\}\subseteq \mathcal D^{\omega^\ast}(P)\subseteq \mathcal D (P)$.

\begin{example}\label{directed not countably-directed} For the countable chain $\mn$ (with the usual order of natural numbers), $(\mn^{(<\omega)}, \subseteq)$ is directed in $2^{\mn}$, but not countably-directed.
\end{example}

\begin{definition}\label{omega ast Scott topology} A subset $U$ of a poset $P$ is $\omega^{\ast}$-\emph{Scott open} if
(i) $U=\mathord{\uparrow}U$ and (ii) for any countably-directed subset $D$ with
$\bigvee D$ existing, $\bigvee D\in U$ implies $D\cap
U\neq\emptyset$. All $\omega^{\ast}$-Scott open subsets of $P$ form a topology,
called the $\omega^{\ast}$-\emph{Scott topology} on $Q$ and
denoted by $\sigma_{\omega^\ast}(Q)$. Let $\Sigma_{\omega^\ast}~\!\! Q=(Q,\sigma_{\omega^\ast}(Q))$.
\end{definition}

Clearly, $\upsilon (P)\subseteq \sigma (P)\subseteq \sigma_{\omega^\ast} (P)\subseteq \alpha (P)$. In general, $\sigma (P)\neq \sigma_{\omega^\ast} (P)$ as shown in Example \ref{omega-d-space needed} below.

It is straightforward to verify the following two results (cf. \cite[Proposition II-2.1]{redbook}).

\begin{proposition}\label{omega ast Scott continuous} A continuous function $f : X \rightarrow Y$ from an $\omega^\ast$-$d$-space $X$ to any $T_0$ space $Y$ preserves countably-directed sups in the specialization orders.
\end{proposition}

\begin{proposition}\label{omega ast Scott continuous=preserves omega ast sups} Let $P, Q$ be posets and $f : P \rightarrow Q$. Then the following two conditions are equivalent:
\begin{enumerate}[\rm (1)]
\item $f : \Sigma_{\omega\ast} P \rightarrow \Sigma_{\omega^\ast} Q$ is continuous.
\item $f$ preserves countably-directed sups, that is, for every $D\in \mathcal D_{\omega^\ast} (P)$ for which $\bigvee D$ exists, $\bigvee f(D)$ exists and $f(\vee D)=\vee f(D)$.
\end{enumerate}
\end{proposition}

\begin{definition}\label{omega ast-d-space} A $T_0$ space $X$ is said to be an $\omega^{\ast}$-$d$-space (or an \emph{$\omega^\ast$-monotone convergence space}), if every subset $D$ countably-directed relative to the specialization order of $X$ has a sup, and the relation $\mathrm{sup}~D\in U$ for any open set $U$ of $X$ implies $D\cap U\neq\emptyset$.
\end{definition}

For a $T_0$ space $X$ (endowed with the specialization order), let $\mathcal D_c^{\omega^{\ast}}(X)=\{\overline{D} : D\in \mathcal D^{\omega^{\ast}}(X)\}$. Then $\mathcal S_c(X)\subseteq \mathcal D_c^{\omega^{\ast}}(X)\subseteq \mathcal D_c(X)$.

\begin{proposition}\label{omega ast d-space charac} For a $T_0$ space $X$, the following conditions are equivalent:
\begin{enumerate}[\rm (1) ]
	        \item $X$ is an $\omega^\ast$-$d$-space.
            \item $\mathcal D_{c}^{\omega^\ast}(X)=\mathcal S_c(X)$, that is, for any $D\in\mathcal D^{\omega^\ast}(X)$, the closure of $D$ has a (unique) generic point.
            \item $X$ (with the specialization order $\leq_X$) is an $\omega^\ast$-dcpo and $\mathcal O(X)\subseteq \sigma_{\omega^\ast}(X)$.
            \item  For any $D\in \mathcal D^{\omega^\ast}(X)$ and $U\in \mathcal O(X)$, $\bigcap\limits_{d\in D}\ua d\subseteq U$ implies $\ua d \subseteq U$ \emph{(}i.e., $d\in U$\emph{)} for some $d\in D$.
            \item  For any $D\in \mathcal D^{\omega^\ast}(X)$ and $A\in \mathcal C(X)$, if $D\subseteq A$, then $A\cap\bigcap\limits_{d\in D}\ua d\neq\emptyset$.
            \item  For any $D\in \mathcal D^{\omega^\ast}(X)$ and $A\in \ir_c(X)$, if $D\subseteq A$, then $A\cap\bigcap\limits_{d\in D}\ua d\neq\emptyset$.
            \item  For any $D\in \mathcal D^{\omega^\ast}(X)$, $\overline{D}\cap\bigcap\limits_{d\in D}\ua d\neq\emptyset$.
\end{enumerate}
\end{proposition}
\begin{proof} (1) $\Rightarrow$ (2): Let $D\in\mathcal D^{\omega^\ast}(X)$. Then by (1), $\bigvee D$ exists and the relation $\bigvee D\in U$ for any open set $U$ of $X$ implies $D\cap U\neq\emptyset$. Therefor, $\overline{D}=\overline{\{\mathrm{sup}~D\}}$.

(2) $\Rightarrow$ (3): For each $D\in \mathcal D^{\omega^\ast}(X)$ and $A\in \mathcal C(X)$ with $D\subseteq A$, by condition (2) there is $x\in X$ such that $\overline{D}=\overline{\{x\}}$, and consequently, $\bigvee D=x$ and $\bigvee D \in A$ since $\overline{D}\subseteq A$. Thus $X$ is an $\omega^\ast$-dcpo and $\mathcal O(X)\subseteq \sigma_{\omega^\ast}(X)$.

(3) $\Rightarrow$ (4): Suppose that $D\in \mathcal D^{\omega^\ast}(X)$ and $U\in \mathcal O(X)$ with $\bigcap\limits_{d\in D}\ua d\subseteq U$. Then by condition (3), $\ua \bigvee D=\bigcap\limits_{d\in D}\ua d\subseteq U\in \sigma_{\omega^\ast} (X)$. Therefore, $\bigvee D\in U$, and whence
$d\in U$ for some $d\in D$.

(4) $\Rightarrow$ (5): If $A\cap\bigcap\limits_{d\in D}\ua d=\emptyset$, then $\bigcap\limits_{d\in D}\ua d \subseteq X\setminus A$. By condition (4), $\ua d\subseteq X\setminus A$ for some $d\in D$, which is in contradiction with $D\subseteq A$.

(5) $\Rightarrow$ (6) $\Rightarrow$ (7): Trivial.

(7) $\Rightarrow$ (1): For each $D\in \mathcal D^{\omega^\ast}(X)$ and $A\in \mathcal C(X)$ with $D\subseteq A$, by condition (6), we have $\overline{D}\cap\bigcap\limits_{d\in D}\ua d\neq\emptyset$. Select an $x\in \overline{D}\cap\bigcap\limits_{d\in D}\ua d$. Then $D\subseteq \da x\subseteq \overline{D}$, and hence $\overline{D}=\da x$. Thus $X$ is an $\omega^\ast$-$d$-space.
\end{proof}

By Proposition \ref{omega ast d-space charac}, every $d$-space is an $\omega$-$d$-space, and for any $\omega^\ast$-dcpo $P$, $\Sigma_{\omega^\ast} P$ is an $\omega^\ast$-$d$-space. Let $Q=(\mn^{(<\omega)}, \subseteq)$. Then $Q$ is an $\omega^\ast$-dcpo but not a dcpo. If $P$ is any $\omega^\ast$-dcpo but not a dcpo, then $\Sigma_{\omega^\ast} (P)$ is an $\omega^\ast$-$d$-space but not a $d$-space.

\begin{definition}\label{omega ast-DC space} A $T_0$ space $X$ is called a \emph{countable}-\emph{directed closure space}, or $\omega^{\ast}$-$\mathbf{DC}$ \emph{space} for short, if $\ir_c (X)=\mathcal D_c^{\omega^{\ast}}(X)$, that is, for each $A\in \ir_c (X)$, there exists a countably-directed subset $D$ of $X$ such that $A =\overline{D}$.
\end{definition}

Now we introduce another type of weak well-filtered spaces.

\begin{definition}\label{omega ast WF space}
	A $T_0$ space $X$ is called $\omega^\ast$-\emph{well}-\emph{filtered}, if for any countable-filtered family $\{K_i:i\in I\}\subseteq \mk (X)$ (that is, $\{K_i:i\in I\}\in \mathcal D^{\omega^\ast}(\mk (X))$) and  $U\in\mathcal O(X)$, it satisfies
		$$\bigcap_{i\in I}K_i\subseteq U \ \Rightarrow \  \exists j\in I, K_j\subseteq U.$$
\end{definition}

Clearly, every well-filtered space is $\omega^\ast$-well-filtered. The converse implications  does not hold in general, as shown by the following example.

\begin{example}\label{omega ast WF is not WF} Let $P=(\mn^{(<\omega)}, \subseteq)$ and $X=\Sigma_{\omega^\ast} P$. It is easy to verify that any countably-directed subset of $P$ has a largest element. Therefore, $\sigma_{\omega^\ast} (P)=\alpha (P)$ and $\mathsf{K}(X)=\{\ua F : F\in \mn^{(<\omega)}\}$, and hence $X$ is $\omega^\ast$-well-filtered as any countable-filtered family $\{K_i:i\in I\}\subseteq \mk (X)$ has a least element. Since $P$ is not a dcpo ($P$ is directed but has not a largest element), $X$ is not a $d$-space, and hence not well-filtered.
\end{example}

\begin{proposition}\label{omega ast WF is omega ast d-space} Every $\omega^\ast$-well-filtered space is an $\omega^\ast$-$d$-space.
\end{proposition}
\begin{proof} Let $X$ be an $\omega^\ast$-well-filtered space and $D\in \mathcal D^{\omega^\ast} (X)$. Then $\{\ua d : d\in D\}\in \mathcal D^{\omega^\ast}(\mk (X))$. By the $\omega^\ast$-well-filteredness of $X$, we have $\bigcap_{d\in D}\ua d\nsubseteq X\setminus \overline{D}$ or, equivalently,  $\bigcap_{d\in D}\ua d\cap\overline{D}\neq \emptyset$. Therefore, there is $x\in \bigcap_{d\in D}\ua d\cap\overline{D}$, and hence $\overline{D}=\overline{\{x\}}$.
\end{proof}

In the following, using the topological Rudin's Lemma, we prove that a $T_0$ space $X$ is  $\omega^\ast$-well-filtered if{}f the Smyth power space of $X$ is $\omega^\ast$-well-filtered if{}f the Smyth power space of $X$ is $\omega^\ast$-$d$-space. The corresponding results for well-filteredness and $\omega$-well-filteredness are proved in \cite{xi-zhao-MSCS-well-filtered, xu-shen-xi-zhao2, xu-shen-xi-zhao1}.

\begin{theorem}\label{Smythomegawf}
	For a $T_0$ space $X$, the following conditions are equivalent:
\begin{enumerate}[\rm (1)]
		\item $X$ is $\omega^\ast$-well-filtered.
        \item $P_S(X)$ is an ${\omega^\ast}$-$d$-space.
        \item $P_S(X)$ is $\omega^\ast$-well-filtered.
\end{enumerate}
\end{theorem}
\begin{proof}

(1) $\Rightarrow$ (2): Suppose that $X$ is an  $\omega^\ast$-well-filtered space. For any countable-filtered family $\mathcal K\subseteq \mk (X)$, by the $\omega^\ast$-well-filteredness of $X$,  $\bigcap \mathcal K \in \mk (X)$. Therefore, by Lemma \ref{sups in Smyth}, $\mk (X)$ is an ${\omega^\ast}$-dcpo. Clearly, by the $\omega^\ast$-well-filteredness of $X$, $\Box U\in \sigma_{\omega^\ast} (\mk (X))$ for any $U\in \mathcal O(X)$. Thus $P_S (X)$ is an ${\omega^\ast}$-$d$-space.

(2) $\Rightarrow$ (3): Suppose that $\{\mathcal K_i : i\in I\}\subseteq \mk(P_S(X))$ is countable-filtered, $\mathcal U\in \mathcal O(P_S(X))$, and $\bigcap\limits_{i\in I} \mathcal K_i \subseteq \mathcal U$. If $\mathcal K_i\not\subseteq \mathcal U$ for all $i\in I$, then by Lemma \ref{t Rudin}, $\mk (X)\setminus \mathcal U$ contains an irreducible closed subset $\mathcal A$ that still meets all $\mathcal K_i$ ($i\in I$). For each $i\in I$, let $K_i=\bigcup \ua_{\mk (X)} (\mathcal A\bigcap \mathcal K_i)$ ($=\bigcup (\mathcal A\bigcap \mathcal K_i$)). Then by Lemma \ref{K union}, $\{K_i : i\in I\}\subseteq \mk (X)$ is countable-filtered, and $K_i\in \mathcal A$ for all $i\in I$ since $\mathcal A=\da_{\mk (X)}\mathcal A$. Let $K=\bigcap\limits_{i\in I} K_i$. Then $K\in \mk (X)$ and $K=\bigvee_{\mk (X)} \{K_i : i\in I\}\in \mathcal A$ by Lemma \ref{sups in Smyth} and condition (2). We claim that $K\in \bigcap\limits_{i\in I}\mathcal K_i$. Suppose, on the contrary, that $K\not\in \bigcap\limits_{i\in I}\mathcal K_i$. Then there is a $j\in I$ such that $K\not\in \mathcal K_j$. Select a $G\in \mathcal A\bigcap \mathcal K_j$. Then $K\not\subseteq G$ (otherwise, $K\in \ua_{\mk (X)}\mathcal K_j=\mathcal K_j$, being a contradiction with $K\not\in \mathcal K_j$), and hence there is a $g\in K\setminus G$. It follows that $g\in K_i=\bigcup (\mathcal A\bigcap \mathcal K_i)$ for all $i\in I$ and $G\not\in \Diamond_{\mk (K)}\overline{\{g\}}$. For each $i\in I$, by $g\in K_i=\bigcup (\mathcal A\bigcap \mathcal K_i)$, there is a $K_i^g\in \mathcal A\bigcap \mathcal K_i$ such that $g\in K_i^g$, and hence $K_i^g\in\Diamond_{\mk (K)}\overline{\{g\}}\bigcap\mathcal A\bigcap \mathcal K_i$. Thus $\Diamond_{\mk (K)}\overline{\{g\}}\bigcap\mathcal A\bigcap \mathcal K_i\neq\emptyset$ for all $i\in I$. By the minimality of $\mathcal A$, we have $\mathcal A=\Diamond_{\mk (K)}\overline{\{g\}}\bigcap\mathcal A$, and consequently, $G\in \mathcal A\bigcap \mathcal K_j=\Diamond_{\mk (K)}\overline{\{g\}}\bigcap\mathcal A\bigcap \mathcal K_j$, which is a contradiction with $G\not\in \Diamond_{\mk (K)}\overline{\{g\}}$. Thus $K\in \bigcap\limits_{i\in I}\mathcal K_i\subseteq \mathcal U\subseteq \mk (X)\setminus \mathcal A$, being a contradiction with $K\in \mathcal A$. Therefore, $P_S(X)$  is $\omega^\ast$-well-filtered.

(3) $\Rightarrow$ (1): Suppose that $\mathcal K\subseteq \mathord{\mathsf K}(X)$ is countable-filtered, $U\in \mathcal O(X)$, and $\bigcap \mathcal K \subseteq U$. Let $\widetilde{\mathcal K}=\{\ua_{\mk (X)}K : K\in \mathcal K\}$. Then $\widetilde{\mathcal K}\subseteq \mk (P_S(X))$ is countable-filtered and $\bigcap \widetilde{\mathcal K} \subseteq \Box U$. By the $\omega^\ast$-well-filteredness of $P_S(X)$, there is a $K\in \mathcal K$ such that $\ua_{\mk (X)}K\subseteq \Box U$, and whence $K\subseteq U$, proving that $X$ is $\omega^\ast$-well-filtered.
\end{proof}

\begin{definition}\label{omega ast DC space}
	A $T_0$ space $X$ is called a {countably-directed closure space}, $\omega^\ast$-$\mathsf{DC}$ \emph{space} for short, if $\ir_c(X)=\mathcal{D}_c^{\omega^\ast}(X)$, that is, for each $A\in \ir_c(X)$, there exists a countably-directed subset of $X$ such that $A=\overline{D}$.
\end{definition}

By Remark \ref{sups in Smyth}, Proposition \ref{omega ast d-space charac} and Proposition \ref{omega ast WF is omega ast d-space}, we get the following result.

\begin{proposition}\label{sober equiv using omega ast DC} For any $T_0$ space $X$, the following conditions are equivalent:
	\begin{enumerate}[\rm (1)]
		\item $X$ is sober.
		\item $X$ is an $\omega^\ast$-$\mathsf{DC}$ and $\omega^\ast$-well-filtered space.
        \item $X$ is an $\omega^\ast$-$\mathsf{DC}$ and $\omega^\ast$-$d$-space.
     \end{enumerate}
\end{proposition}

\section{First-countability of sobrifications and $\omega$-Rudin spaces}

In this section, we prove that if the sobrification of a $T_0$ space $X$ is first-countable, then $X$  is a $\omega$-Rudin space. Hence every $\omega$-well-filtered space having a first-countable sobrification is sober.

We first prove  two useful lemmas.

\begin{lemma}\label{countable minimal compact} Let $X$ be a $T_0$ space and $A\in \mathcal C  (X)$. For $\{U_n : n\in \mn \}\subseteq \mathcal O(X)$ with $U_1\supseteq U_2\supseteq ... \supseteq U_n\supseteq U_{n+1}\supseteq ...$, if $A\in m(\{U_n : n\in \mn \}$ and $x_n\in U_n \cap A$ for each $n\in \mn$, then every subset of $\{x_n : n\in \mn\})$ is compact.
\end{lemma}

\begin{proof} Suppose $E\subseteq\{x_n : n\in \mn\}$ and $\{V_i:i\in I\}$ is an open cover of $E$, that is, $E\subseteq \bigcup_{i\in I}V_i$.

{Case 1.} $E\cap (X\setminus V_j)$ is finite for some $j\in I$.

Then there is $I_j\in I^{(<\omega)}$ such that $E\cap (X\setminus V_j)\subseteq \bigcup_{i\in I_j}V_i$, and hence $E\subseteq V_j \cup\bigcup_{i\in I_j}V_i$.

{Case 2.} $E\cap (X\setminus V_i)$ is infinite for all $i\in I$.

For each $i\in I$, since $U_1\supseteq U_2\supseteq ... \supseteq U_n\supseteq U_{n+1}\supseteq ...$ and $x_n\in U_n \cap A$ for each $n\in \mn$, we have that $U_n\cap A\cap (X\setminus V_i)\neq\emptyset$ for all $n\in \mn$, and hence $A\cap (X\setminus V_i)=A$ by the minimality of $A$. It follows that $A\subseteq \bigcap_{i\in I}(X\setminus V_i)=X\setminus \bigcup_{i\in I}V_i$. Therefore, $E\subseteq A\cap \bigcup_{i\in I}V_i=\emptyset$.

By Case 1 and Case 2, $E$ is compact.
\end{proof}

\begin{lemma}\label{local base compact} Let $X$ be a $T_0$ space and $A\in \ir_c (X)$. For any open neighborhood base $\{\Diamond U_i : i\in I\}$ of $A$ in $X^s$, $A\in m(\{U_i : i\in I\})$.
\end{lemma}
\begin{proof} Clearly, $A\in M(\{U_i : i\in I\})$. Suppose $B\in \mathcal C  (X)$ and $B\subseteq A$. If $B\neq A$, then $A\cap (X\setminus B)\neq \emptyset$, and hence $A\in \Diamond (X\setminus B)$. Since $\{\Diamond U_i : i\in I\}$ is an open neighborhood base at $A$ in $X^s$, there is $j\in I$ such that $\Diamond U_j\subseteq \Diamond (X\setminus B)$ or, equivalently, $U_j\subseteq X\setminus B$. Therefore, $U_j\cap B=\emptyset$. So $B\notin M(\{U_i : i\in I\})$. Thus $A\in m(\{U_i : i\in I\})$.
\end{proof}

\begin{proposition}\label{second-countable space} For a $T_0$ space $X$, the following two conditions are equivalent:
\begin{enumerate}[\rm (1)]
\item $X$ is second-countable.
\item $X^s$ is second-countable.
\end{enumerate}
\end{proposition}
\begin{proof} For any  $\{U_i: i\in I\}\subseteq \mathcal O(X)$, it is easy to verify that $\{U_i: i\in I\}$ is a base of $X$ iff $\{\Diamond U_i : i\in I\}$ is a base of $X^s$.
\end{proof}

Since the first-countability is a hereditary property, by Remark \ref{eta continuous}, we have the following result.

\begin{proposition}\label{soberification first-countable implies X is also} For a $T_0$ space $X$, if $X^s$ is first-countable, then $X$ is first-countable.
\end{proposition}

The converse of Proposition \ref{soberification first-countable implies X is also} does not hold in general, as shown in Example \ref{first-countable omega WF is not sober} below.

\begin{proposition}\label{countable first-countable} If a $T_0$ space $X$ is first-countable and $| X |\leq \omega$, then $X^s$ is first-countable.
\end{proposition}

\begin{proof} Let $X=\{x_n : n\in \mn\}$. For each $n\in \mn$, since $X$ is first-countable, there is a countable base $\{U_m(x_n) : m\in \mn\}$ at $x_n$. For any $A\in \ir_c (X)$, it straightforward to verify that $\{\Diamond U_m(x_n) : (m, n)\in  \mn\times \mn \mathrm{~and~} A\cap U_m(x_n)\neq \emptyset \}$ is a countable base at $A$ in $X^s$. Thus $X^s$ is first-countable.
\end{proof}

The following example shows that the Scott topology on a countable complete lattice may not be first-countable.
	
\begin{example}\label{countable poset but Scott topology not first-countable} \emph{(\cite{xu-shen-xi-zhao2})}
	Let $L=\{\bot\}\cup(\mn \times \mn)\cup\{\top\}$ and define a partial order $\leq$ on $L$ as follows:
	\begin{itemize}
		\item [(i)] $\forall (n,m)\in \mn\times\mn$, $\bot\leq (n,m) \leq\top$;
		\item [(ii)] $\forall (n_1,m_1), (n_2,m_2)\in\mn\times\mn$, $(n_1,m_1)\leq(n_2,m_2)$ iff $n_1=n_2$ and $m_1\leq m_2$.
	\end{itemize}

 Then $(L, \sigma (L))$ does not have any countable base at $\top$. Assume, on the contrary,  there exists a countable base $\{U_n : n\in\mn\}$ at $\top$. Then for each $n\in\mn$, as
$$\bigvee (\{n\}\times\mn)=\top\in U_n,$$
there exists $m_n\in\mn$ such that $(n,m_n)\in U_n$.
Let $U=\bigcup_{n\in\mn }\ua (n,m_n+1)$. Then $U\in \sigma (L)$. But for each $n\in N$, $(n,m_{n})\in U_n\setminus U$, which contradicts that  $\{U_n:n\in\mn\}$ is a base at $\top$. Therefore, $(L, \sigma (L))$ is not first-countable. One can easily check that $(L, \sigma (L))$ is sober.
\end{example}

\begin{theorem}\label{sobrifcaltion first-countable is Rudin} For a $T_0$ space $X$, if $X^s$ is first-countable, then $X$ is an $\omega$-Rudin space.
\end{theorem}
\begin{proof} Let $A\in \ir_c (X)$. By the first-countability of $X^s$, there is an open neighborhood base $\{\Diamond U_n : n\in \mn\}$ of $A$ such that

$$\Diamond U_1\supseteq \Diamond U_2\supseteq\ldots\supseteq \Diamond U_n\supseteq\ldots, $$

\noindent or equivalently, $U_1\supseteq U_2\supseteq\ldots\supseteq U_n\supseteq\ldots$. By Lemma \ref{local base compact}, $A\in m(\{U_n : n\in \mn\}$. For each $n\in \mn$, choose an $x_n\in U_n\cap A$, and let $K_n=\ua \{x_m : m\geq n\}$. Then $K_1\supseteq K_2\supseteq\ldots\supseteq  K_n\supseteq\ldots$, and $\{K_n : n\in \mn \}\subseteq \mathsf{K}(X)$ by Lemma \ref{countable minimal compact}. Clearly, $A\in M(\{K_n : n\in\mn\})$. For any $B\in \mathcal C (X)$, if $B$ is a proper subset of $A$, that is, $A\cap (X\setminus B)=A\setminus B\neq \emptyset$, then $A\in \Diamond (X\setminus B)\in \mathcal O(X^s)~(=\mathcal O(P_H(\ir_c(X))))$. Therefore, $\Diamond U_m\subseteq \Diamond (X\setminus B)$ for some $m\in \mn$, and hence $U_m\subseteq X\setminus B$ or, equivalently, $U_m\cap B=\emptyset$. Thus $B\notin M(\{K_n : n\in\mn\})$, proving $A\in m(\{K_n : n\in\mn\})$. So $X$ is an $\omega$-Rudin space.
\end{proof}

\begin{corollary}\label{second-countable is Rudin} Every second-countable $T_0$ space is an $\omega$-Rudin space.
\end{corollary}

\begin{corollary}\label{countable first-countable is Rudin} Every countable first-countable $T_0$ space is an $\omega$-Rudin space.
\end{corollary}

\begin{theorem}\label{sobrification first-countable and omega-WF is sober} Every $\omega$-well-filtered space with a first-countable sobrification is sober
\end{theorem}
\begin{proof} For $A\in\ir_c(X)$, by Theorem \ref{sobrifcaltion first-countable is Rudin} and its proof (or Proposition \ref{omega WF charact}), there is a decreasing sequence $\{K_n : n\in\mn \}\subseteq \mathsf{K}(X)$ such that $A\in m(\{K_n n\in \mn\}$. Since $X$ is $\omega$-well-filtered, $\bigcap_{n\in\mn}K_n\nsubseteq X\setminus A$, that is, $\bigcap_{n\in\mn}K_n\cap A\neq\emptyset$. Choose $x\in \bigcap_{n\in\mn}K_n\cap A$. Then $\overline{\{x\}}\in M(\{K_n n\in \mn\}$ and $\overline{\{x\}}\subseteq A$. By the minimality of $A$, we have $A=\overline{\{x\}}$. Thus $X$ is sober.
\end{proof}

\begin{corollary}\label{secound-countable omega WF is sober} Every second-countable $\omega$-well-filtered space is sober.
\end{corollary}

\begin{corollary}\label{countable first-countable omega WF is sober} Every countable first-countable $\omega$-well-filtered space is sober.
\end{corollary}

In Theorem \ref{sobrifcaltion first-countable is Rudin} and Theorem \ref{sobrification first-countable and omega-WF is sober}, the first-countability of $X^s$ can not be weakened to that of $X$ as shown in the following example (see also Example \ref{omega-d-space needed} in Section 7).

\begin{example}\label{first-countable omega WF is not sober}  Let $\omega_1$ be the first uncountable ordinal number and $P=[0, \omega_1)$. Then
\begin{enumerate}[\rm (a)]
\item $\mathcal C(\Sigma P)=\{\da t : t\in P\}\cup\{\emptyset, P\}$.
\item $\Sigma P$ is compact since $P$ has a least element $0$.
\item $\Sigma P$ is first-countable.
\item $(\Sigma P)^s$ is not first-countable. In fact, it is easy to verify that $(\Sigma P)^s$ is homeomorphic to $\Sigma [0, \omega_1]$. Since sup of a countable family of countable ordinal numbers is still a countable ordinal number, $\Sigma [0, \omega_1]$ has no countable base at the point $\omega_1$.
    \item $P$ is an $\omega$-dcpo but not a dcpo. So $\Sigma P$ is an $\omega$-$d$-space but not a $d$-space, and hence not a sober space.
    \item $\mathsf{K}(\Sigma P)=\{\ua x : x\in P\}$. For $K\in \mathsf{K}(\Sigma P)$, we have $\mathrm{inf}~K\in K$, and hence $K=\ua \mathrm{inf}~K$.
    \item $\Sigma P$ is a Rudin space. One can easily check that $\ir_c(\Sigma P)=\{\downarrow x : x\in P\}\cup \{P\}$. Clearly, $\downarrow x$ is a Rudin set for each $x\in P$. Now we show that $P$ is a Rudin set. First, $\{\ua s : s\in P\}$ is filtered. Second, $P\in M(\{\ua s : s\in P\})$. For a closed subset $B$ of $\Sigma P$, if $B\neq P$, then $B=\da t$ for some $t\in P$, and hence $\ua (t+1)\cap\da t=\emptyset$. Thus $B\notin M(\{\ua s : s\in P\})$, proving that $P$ is a Rudin set.
        \item $\Sigma P$ is not an $\omega$-Rudin space. We prove that the irreducible closed set $P$ is not an $\omega$-Rudin set. For any countable filtered family $\{\ua \alpha_n : n\in\mn \}\subseteq \mathsf{K}(\Sigma P)$, let $\beta=\mathrm{sup}\{\alpha_n : n\in\mn\}$. Then $\beta$ is still a countable ordinal number. Clearly, $\da \beta \in M(\{\ua \alpha_n : n\in \mn\}$ and $P\neq \da \beta$. Therefore, $P\notin m(\{\ua \alpha_n : n\in \mn\})$. Thus $P$ is not an $\omega$-Rudin set, and hence $\Sigma P$ is not an $\omega$-Rudin space.
    \item $\Sigma P$ is $\omega$-well-filtered. If $\{\ua x_n : n\in \mn\}\subseteq \mathsf{K}(\Sigma P)$ is countable filtered family and $U\in \sigma (P)$ with $\bigcap_{n\in \mn}\ua x_n\subseteq U$, then $\{x_n : i\in\mn\}$ is a countable subset of $P=[0, \omega_1)$. Since sup of a countable family of countable ordinal numbers is still a countable ordinal number, we have $\beta=\mathrm{sup}\{x_n : n\in \mn\}\in P$, and hence $\ua \beta=\bigcap_{n\in \mn}\ua x_n\subseteq U$. Therefore, $\beta\in U$, and consequently, $x_n\in U$ for some $n\in \mn$ or, equivalently, $\ua x_n\subseteq U$, proving that $\Sigma P$ is $\omega$-well-filtered.
\end{enumerate}
\end{example}

\section{First-countability and well-filtered determined spaces}

In this section, we show that any first-countable $T_0$ space is well-filtered determined. In \cite{xu-shen-xi-zhao2} it was shown that in a first-countable $\omega$-well-filtered $T_0$ space $X$, all irreducible closed subsets of $X$ are directed (see \cite[Theorem 4.1]{xu-shen-xi-zhao2}). In the following we will strengthen this result by proving that in a first-countable $\omega$-well-filtered space $X$, every irreducible closed subset of $X$ is countably-directed.

\begin{lemma}\label{first-countable omega-directed} Suppose that $X$ is a first-countable $T_0$ space, $Y$ is an $\omega$-well-filtered space and $f: X \rightarrow Y$ is a continuous mapping. Then for any $A\in \ir (X)$ and $\{a_n : n\in \mn\}\subseteq \overline{A}$, $\bigcap_{n\in \mn} \uparrow f(a_n) \bigcap \overline{f(A)}\neq \emptyset$.
\end{lemma}
\begin{proof}  For each $x\in X$, since $X$ is first-countable, there is an open neighborhood base $\{U_n(x) : n\in \mn\}$ of $x$ such that

$$U_1{(x)}\supseteq U_2{(x)}\supseteq\ldots\supseteq U_k{(x)}\supseteq\ldots, $$

\noindent that is,  $\{U_n(x) : n\in \mn\}$ is a decreasing sequence of open subsets.

 For each $(n, m)\in \mn\times \mn$, since $a_n\in \overline{A}$ and $A\in \ir (X)$, we have $A\cap \bigcap\limits_{i=1}^{m}U_{l_i}(a_{k_i})\cap A\neq\emptyset$ for all $\{(l_i, k_i)\in \mn\times \mn : 1\leq i\leq m\}$.

Choose $c_1\in  U_1{(a_1)}\cap \cap A$.
Now suppose we already have a set $\{c_1,c_2,\ldots, c_{n-1}\}$ such that
for each $2\leq i\leq n-1$,
$$c_{i}\in \bigcap\limits_{j=1}^{i-1}U_{i}(c_j)\cap\bigcap\limits_{j=1}^{i} U_{i}(a_j)\cap A.$$
Note that above condition implies that for any positive integer $1\leq k\leq n-1$,
$$\bigcap\limits_{j=1}^{k}U_{k}(c_j)\cap\bigcap\limits_{j=1}^{k} U_{i}(a_j)\cap A\neq \emptyset,\mathrm{~and~}$$
 $$\bigcap\limits_{j=1}^{n-1}U_{n}(c_j)\cap\bigcap\limits_{j=1}^{n} U_{n}(a_j)\cap A\neq \emptyset.$$
So we can choose $c_{n}\in \bigcap\limits_{j=1}^{n-1}U_{n}(c_j)\cap\bigcap\limits_{j=1}^{n} U_{n}(a_j)\cap A$. By $A\in \ir (X)$ again, we have
$$\bigcap\limits_{j=1}^{n}U_{n}(c_j)\cap\bigcap\limits_{j=1}^{n} U_{n}(a_j)\cap A\neq \emptyset.$$

By induction, we can obtain a set $\{c_n: n\in\mn\}$.

Let $K_n=\ua\{c_k:k\geq n\}$ for each $n\in\mn$.

{\bf Claim 1:} $\forall n\in\mn$, $K_n\in \mathsf{K}(X)$.

Suppose $\{V_i:i\in I\}$ is an open cover of $K_n$, i.e., $K_n\subseteq \bigcup_{i\in I}V_i$. Then there is $i_0\in I$ such that $c_n\in V_{i_0}$, and thus there is $m\geq n$ such that $c_n\in U_m(c_n)\subseteq V_{i_0}$. It follows that $c_k\in U_m(c_n)\subseteq V_{i_0}$ for all $k\geq m$. Thus $\{c_k: k\geq m\}\subseteq V_{i_0}$. For each $c_k$, where $n\leq k< m$, choose a $V_{i_k}$ such that $c_k\in V_{i_k}$.
Then the finite family $\{V_{i_k}:n\leq k< m\}\cup \{V_{i_0}\}$  covers $K_n$. So $K_n$ is compact.

{\bf Claim 2:} $\{\ua f(K_n) : n\in \mn\}\subseteq \mathsf{K}(Y)$ and $\ua f(K_1)\supseteq \ua f(K_2) \supseteq ... \supseteq \ua f(K_n) \supseteq \ua f(K_{n+1})\supseteq ...$

For each $n\in \mn$, since $K_m\in \mathsf{K}(X)$ and $f$ is continuous, we have $\ua f(K_m)\in \mathsf{K}(Y)$. Clearly, $K_1\supseteq K_2 \supseteq ... \supseteq K_n \supseteq K_{n+1}\supseteq ...$, and hence $\ua f(K_1)\supseteq \ua f(K_2) \supseteq ... \supseteq \ua f(K_n) \supseteq \ua f(K_{n+1})\supseteq ...$

{\bf Claim 3:} $\bigcap_{n\in \mn}\ua f(K_n)=\bigcap_{n\in \mn}\ua f(c_n)$.

Clearly, $\bigcap_{n\in \mn}\ua f(c_n)\subseteq \bigcap_{n\in \mn}\ua f(K_n)$. Now we show $\bigcap_{n\in \mn}\ua f(K_n)\subseteq\ua f(c_m)$ for all $m\in \mn$. Suppose $V\in \mathcal O(Y)$ with $f(c_m)\in V$. Then $c_m\in f^{-1}(V)\in \mathcal O(X)$, and whence $U_{n(m)}(c_m)\subseteq f^{-1}(V)$ for some $n(m)\in \mn$. For any $l\geq \mathrm{max}\{m, n(m)\}+1$, we have $K_l\subseteq U_l(c_m)\subseteq U_{n(m)}(c_m)\subseteq f^{-1}(V)$, and consequently, $\ua f(K_l)\subseteq V$. It follows that $\bigcap_{n\in \mn}\ua f(K_n)\subseteq \bigcap_{f(c_m)\in V\in \mathcal O(Y)} V=\ua f(c_m)$. Thus $\bigcap_{n\in \mn}\ua f(K_n)\subseteq \bigcap_{n\in \mn}\ua f(c_n)$.

{\bf Claim 4:} $\bigcap_{n\in \mn}\ua f(K_n)\subseteq \bigcap_{n\in \mn}\ua f(a_n)$.

For $m\in \mn$ and $W\in \mathcal O(Y)$ with $f(a_m)\in W$. Then $a_m\in f^{-1}(W\in \mathcal O(X)$, and whence $U_{k(m)}(a_m)\subseteq f^{-1}(W)$ for some $k(m)\in \mn$. For any $l\geq \mathrm{max}\{m, k(m)\}$, we have $K_l\subseteq U_l(a_m)\subseteq U_{k(m)}(a_m)\subseteq f^{-1}(W)$, and consequently, $\ua f(K_l)\subseteq W$. It follows that $\bigcap_{n\in \mn}\ua f(K_n)\subseteq \bigcap_{f(a_m)\in W\in \mathcal O(Y)} W=\ua f(a_m)$. Thus $\bigcap_{n\in \mn}\ua f(K_n)\subseteq \bigcap_{n\in \mn}\ua f(a_n)$.

{\bf Claim 5:} $\bigcap_{n\in \mn}\ua f(K_n)\in \mathsf{K}(Y)$ and $\bigcap_{n\in \mn}\ua f(K_n)\cap \overline{f(A)}\neq \emptyset$.

By Claim 2 and the $\omega$-well-filteredness of $Y$, $\bigcap_{n\in \mn}\ua f(K_n)\in \mathsf{K}(Y)$. Now we show $\bigcap_{n\in \mn}\ua f(K_n)\cap \overline{f(A)}\neq \emptyset$. Assume, on the contrary, that $\bigcap_{n\in \mn}\ua f(K_n)\cap \overline{f(A)}=\emptyset$ or, equivalently, $\bigcap_{n\in \mn}\ua f(K_n)\subseteq Y\setminus\overline{f(A)}$. Then by the $\omega$-well-filteredness of $Y$ and Claim 2, $\ua f(K_n)\subseteq Y\setminus \overline{f(A)}$, which is in contradiction with $\{c_m : m\geq n\}\subseteq A\cap K_n$. Therefore, $\bigcap_{n\in \mn}\ua f(K_n)\cap \overline{f(A)}\neq \emptyset$.
\end{proof}

\begin{corollary}\label{first-countable omega WF is omega-directed} In a first-countable $\omega$-well-filtered space $X$, every irreducible closed subset of $X$ is countably-directed. Therefore, $X$ is an $\omega^\ast$-$\mathbf{DC}$ space.
\end{corollary}

By Remark \ref{sober implies WF implies d-space}, Proposition \ref{omega ast d-space charac} and Corollary \ref{first-countable omega WF is omega-directed}, we get the following result.

\begin{theorem}\label{first-countable omega-WF omega-d-space}
	For a first-countable $T_0$ space $X$, the following conditions are equivalent:
\begin{enumerate}[\rm (1)]
	\item $X$ is a sober space.
	\item $X$ is a well-filtered space.
	\item $X$ is an $\omega$-well-filtered $d$-space.
    \item $X$ is an $\omega$-well-filtered $\omega^\ast$-$d$-space.
\end{enumerate}
\end{theorem}

A first-countable $d$-space may not be sober as shown in the following example.

\begin{example}\label{first-countable d-space is not sober}
	Let $X$ be a countably infinite set and $X_{cof}$ the space equipped with the \emph{co-finite topology} (the empty set and the complements of finite subsets of $X$ are open). Then
\begin{enumerate}[\rm (a)]
    \item $\mathcal C(X_{cof})=\{\emptyset, X\}\cup X^{(<\omega)}$, $X_{cof}$ is $T_1$ and hence a $d$-space.
    \item $\mk (X_{cof})=2^X\setminus \{\emptyset\}$.
    \item $X_{cof}$ is first-countable.
    \item $X_{cof}$ is locally compact and hence a Rudin space by Theorem \ref{LCrudin}.
    \item $X_{cof}$ is non-sober. $\mathcal{K}_X=\{X\setminus F : F\in X^{(<\omega)}\}\subseteq\mk (X_{cof})$ is countable filtered and $\bigcap \mathcal{K}_X=X\setminus \bigcup X^{(<\omega)}=X\setminus X=\emptyset$, but $X\setminus F\neq\emptyset$ for all $ F\in X^{(<\omega)}$. Thus $X_{cof}$ is not $\omega$-well-filtered, and consequently, $X_{cof}$ is not sober by Theorem \ref{first-countable d-space is not sober}.
\end{enumerate}
\end{example}

The following example shows that a first-countable $\omega$-well-filtered space need not to be sober.

\begin{example}\label{omega-d-space needed}
	Let $L$ be the complete chain $[0,\omega_1]$. Then
\begin{enumerate}[\rm (a)]
\item $\Sigma_{\omega}L$ is a first-countable.
\item $\sigma (P)\neq \sigma_{\omega} (L)$. Since sups of all countable families of countable ordinal numbers are still countable ordinal numbers, we have that $\{\omega_1\}\in\sigma_{\omega}(L)$ but $\{\omega_1\}\notin\sigma (L)$ (note that $\omega_1=\mathrm{sup}~[0, \omega_1)$).

\item $\sigma (L)\neq \sigma_{\omega^\ast} (L)$. It is easy to check that $[\omega, \omega_1]\in \sigma_{\omega^\ast} (L)$ but $[\omega, \omega_1]\notin \sigma (L)$ (note that $\omega=\mathrm{sup}~\mn$).
\item $\mathsf{K}(\Sigma_{\omega}L)=\{\ua \alpha : \alpha\in [0,\omega_1]\}$. For $K\in \mathsf{K}$, we have $\mathrm{inf}~K\in K$, and hence $K=\ua \mathrm{inf}~K$.
 \item $\Sigma_{\omega}L$ is not an $\omega$-Rudin space. It is easy to check that $[0,\omega_1)\in\ir_c (\Sigma_{\omega}L)$ (note that $\{\omega_1\}\in\sigma_{\omega}(L)$). If $[0,\omega_1)\in \mathsf{RD}_\omega(\Sigma_{\omega}L)$, then by (d), there is a countable subset $\{\alpha_n :n\in \mn\}\subseteq [0, \omega_1)$ such that $[0,\omega_1)\in m(\{\ua \alpha_n :n\in \mn\})$. Let $\beta=\mathrm{sup}\{\alpha_n :n\in \mn\}$. Then $\beta\in [0,\omega_1)$, and hence $\da \beta\in \mathcal{C}(\Sigma_{\omega}L)$ and $\da \beta\in M(\{\ua \alpha_n :n\in \mn\})$, which is in contradiction with $[0,\omega_1)\in m(\{\ua \alpha_n :n\in \mn\})$.
    \item $\Sigma_{\omega}L$ is $\omega$-well-filtered. Suppose that $\{\ua \alpha_n : n\in \mn\}\subseteq \mathsf{K}(\Sigma_{\omega}L)$ is countable filtered and $U\in \sigma_{\omega}(L)$ with $\bigcap_{n\in \mn} \ua \alpha_n\subseteq U$. Let $\alpha=\mathrm{sup}\{\alpha_n : n\in\mn\}$. Then $\{\alpha_n : n\in\mn\}$ is a countable directed subset of $L$ and $\alpha\in U$ since $\ua \alpha=\bigcap_{n\in \mn} \ua \alpha_n\subseteq U$. It follows that $\alpha_n\in U$ or, equivalently, $\ua \alpha_n\subseteq U$ for some $n\in \mn$. Thus $\Sigma_{\omega}L$ is $\omega$-well-filtered, and hence an $\omega$-$d$-space.
 \item $\Sigma_{\omega}L$ is not well-filtered. $\{\ua t : t\in [0, \omega_1)\}\subseteq \mathsf{K}(\Sigma_{\omega}L)$ is filtered and $\bigcap_{t\in [0, \omega_1)}\ua t=\{\omega_1\}\in \sigma_{\omega}(L)$, but $\ua t\nsubseteq \{\omega_1\}$ for all $t\in [0, \omega_1)$. Therefore, $\Sigma_{\omega}L$ is not well-filtered.
\item $\Sigma_\omega L$ is not a $d$-space. $[0,\omega_1)\in\ir_c (\Sigma_{\omega}L)$ and $[0,\omega_1)$ is directed, but $[0,\omega_1)\neq\mathrm{cl}_{\sigma_\omega(L)}\{\alpha\}=[0, \alpha]$ for all $\alpha\in L$. Thus $\Sigma_\omega L$ is not a $d$-space.
\item $\Sigma_\omega L$ is not an $\omega^\ast$-$d$-space. In fact, by (b), $\{\omega_1\}\in\sigma_{\omega}(L)$ but $\{\omega_1\}\notin\sigma_{\omega^\ast}(L)$, and hence by Theorem \ref{omega ast d-space charac}, $\Sigma_{\omega}L$ is not an $\omega^\ast$-$d$-space.
  \end{enumerate}

 Since $\Sigma_{\omega}L$ is not well-filtered, it is non-sober. So in Theorem \ref{first-countable omega-WF omega-d-space}, condition (4) (and so condition (3)) cannot be weakened to the condition that $X$ is only an $\omega$-well-filtered space.	

\end{example}

 \begin{definition}\label{countable generating compact saturated subset} Let $X$ be a first-countable $T_0$ space, $A\in \ir (X)$ and $\{a_n : n\in\mn\}$.  The countable family $\{K_n : n\in \mn\}\subseteq \mathsf{K}(X)$ obtained in the proof of Lemma \ref{first-countable omega-directed} is called a decreasing sequence of compact saturated subsets related to $\{a_n : n\in \mn\}$.
 \end{definition}

\begin{theorem}\label{first-countable image is Rudin}  Suppose that $X$ is a first-countable $T_0$ space, $Y$ is an $\omega$-well-filtered space and $f: X \rightarrow Y$ is a continuous mapping. Then for any $A\in \ir (X)$, $\overline{f(A)}\in \mathbf{RD}(Y)$.
\end{theorem}
\begin{proof}  Let $\mathcal{K}_A=\{\bigcap_{n\in \mn}\ua f(K_n) : \{K_n : n\in \mn\}$ is a decreasing sequence of compact saturated subsets related to a countable set $\{a_n : n\in \mn\}\subseteq A\}$. Then by the proof of Lemma \ref{first-countable omega-directed}, we have

{$\mathbf{1}^{\circ}$}  $\mathcal{K}_A\neq\emptyset$ and $\mathcal{K}_A\subseteq \mathsf{K}(Y)$.

{$\mathbf{2}^{\circ}$}  $\ua f(a)\in\mathcal{K}_A$ for all $a\in A$.

For $\{a_n : n\in\mn\}\subseteq A$ with $a_n\equiv a$, as carrying out in the proof of Lemma \ref{first-countable omega-directed}, choose $c_n\equiv a$ for all $n\in \mn$. Then $K_1=K_2= ... =K_n=...=\ua a$, and hence $\ua f(a)=\ua f(\ua a)=\bigcap_{n\in \mn}\ua f(K_n)\in \mathcal{K}_A$.

{$\mathbf{3}^{\circ}$}  $\mathcal{K}_A$ is filtered.

Suppose that $\{K_n : n\in \mn\}$ and $\{G_n : n\in \mn\}$ are decreasing sequences of compact saturated subsets related to countable sets $\{a_n : n\in \mn\}\subseteq A$ and $\{b_n : n\in \mn\}\subseteq A$, respectively. By the proof of Lemma \ref{first-countable omega-directed}, for each $n\in \mn$, $K_n=\ua \{c_m : m\geq n\}$ and $G_n=\ua \{d_m : m\geq n\}$, where $\{c_n : n\in\mn\}\subseteq A$ and $\{d_n : n\in\mn \}\subseteq A$ are obtained by the choice procedures (by induction) in the proof of Lemma \ref{first-countable omega-directed} related to $\{a_n : n\in \mn\}\subseteq A$ and $\{b_n : n\in \mn\}\subseteq A$, respectively. Consider $\{s_n : n\in\mn \}=\{c_1, d_1, c_2, d_2, ..., c_n, d_n, ...\}\subseteq A$, that is,
$$s_n=
	\begin{cases}
	c_k& n=2k+1\\
	d_k& n=2k.
	\end{cases}$$

Then by the proof of Lemma \ref{first-countable omega-directed}, we can inductively choose a countable set $\{t_n : n\in \mn\}$ such that

$$t_{n}\in \bigcap\limits_{j=1}^{n-1}U_{n}(t_j)\cap\bigcap\limits_{j=1}^{n} U_{n}(s_j)\cap A \mathrm{~for all~} n\in\mn.$$

For each $n\in \mn$, let $H_n=\ua \{t_m : m\geq n\}$. Then by Claim 1 and Claim 2 in the proof of Lemma \ref{first-countable omega-directed}, $\ua f(H_n)\in \mathcal{K}_A$. By Claim 3 and Claim 4 in the proof of Lemma \ref{first-countable omega-directed}, we have $\bigcap_{n\in \mn}\ua f(H_n)=\bigcap_{n\in \mn}\ua f(t_n)\subseteq \bigcap_{n\in \mn}\ua f(s_n)=\bigcap_{n\in \mn}\ua f(c_n)\cap \bigcap_{n\in \mn}\ua f(d_n)=\bigcap_{n\in \mn}\ua f(K_n)\cap \bigcap_{n\in \mn}\ua f(G_n)$. Thus $\mathcal{K}_A$ is filtered.

{$\mathbf{4}^{\circ}$}  $\overline{f(A)}\in M(\mathcal{K}_A)$.

By Claim 5 in the proof of Lemma \ref{first-countable omega-directed}, $\overline{f(A)}\in M(\mathcal{K}_A)$.

{$\mathbf{5}^{\circ}$}  $\overline{f(A)}\in m(\mathcal{K}_A)$.

If $B$ is a closed subset with $B\in M(\mathcal{K}_A)$, then for each $a\in A$, by $2^{\circ}$, we have $\ua f(a)\cap B\neq\emptyset$, and hence $f(a)\in B$. It follows that $f(A)\subseteq B$. Thus $\overline{f(A)}\in m(\mathcal{K}_A)$.

By $1^{\circ}$, $3^{\circ}$ and $5^{\circ}$, $\overline{f(A)}\in \mathbf{RD}(Y)$.
\end{proof}

\begin{theorem}\label{first-countable is WD} Every first-countable $T_0$ space is a well-filtered determined space.
\end{theorem}
\begin{proof} Let $X$ be a first-countable $T_0$ space and $A\in \ir_c (X)$. We need to show $A\in\wdd (X)$. Suppose that $ f : X\longrightarrow Y$ is a continuous mapping from $X$ to a well-filtered space $Y$. By Theorem \ref{first-countable image is Rudin}, $\overline{f(A)}\in \mathbf{RD}(Y)$, and whence by the well-filteredness of $Y$ and Proposition \ref{DRWIsetrelation}, there is a (unique) $y_A\in Y$ such that $\overline{f(A)}=\overline{\{y_A\}}$. Thus $A\in \wdd (X)$.
\end{proof}

In \cite[Example 4.15]{liu-li-wu}, a well-filtered space $X$ but not a Rudin space was given. It is straightforward to check that $X$ is not first-countable.

Finally, based on  Theorem \ref{sobrifcaltion first-countable is Rudin} and Theorem \ref{first-countable is WD}, we  pose the following two natural problems.

\begin{problem}\label{first-countable Rudin}
Is every first-countable $T_0$ space a Rudin space?
\end{problem}

\begin{problem}\label{first-countable omega-well-filtered determined}
Is every first-countable $T_0$ space $\omega$-well-filtered determined?
\end{problem}

\end{document}